\documentclass[12pt]{article}
\usepackage[cp1250]{inputenc}
\usepackage[english]{babel}
\usepackage{amsmath,amsfonts,amsthm,amssymb,mnsymbol}
\usepackage{tikz}
\usepackage{tikz-qtree}
\usepackage{url}
\usepackage{bussproofs}
\usepackage[backend=bibtex]{biblatex}
\bibliography{OnGLP}
\pdfoutput=1

\author{Daniyar Shamkanov\thanks{Supported by RFBR (11-01-00281-a, 11-01-00947-a, 12-01-00888-a) and  the  Program of Support for Leading Scientific Schools of Russia (NSh-5593.2012.1).}\\ \normalsize{\textit{Steklov Mathematical Institute of the Russian Academy of Sciences}}\\
\normalsize{\textit{National
Research University Higher School of Economics}}\\
\normalsize{\textit{daniyar.shamkanov@gmail.com}}\\
}
\title{Nested Sequents for Provability Logic $\mathsf{GLP}$}

\newtheorem{thm}{Theorem}[section]
\newtheorem{prop}[thm]{Proposition}
\newtheorem{lem}[thm]{Lemma}
\newtheorem{cor}[thm]{Corollary}

\theoremstyle{remark}

\EnableBpAbbreviations

\begin{document}
\maketitle

\begin{abstract}
We present a proof system for the provability logic $\mathsf{GLP}$ in the formalism of nested sequents and prove the cut elimination theorem for it. As an application, we obtain the reduction of $\mathsf{GLP}$ to its important fragment called $\mathsf{J}$ syntactically.
\end{abstract}



The polymodal provability logic $\mathsf{GLP}$ introduced by G.~Japaridze \cite{Jap86} is a well-known modal logic, which has important applications in proof theory and ordinal analysis of arithmetic \cite{Bek05}. This logic is complete w.r.t. the arithmetical semantics where modalities correspond to reflection principles of restricted logical complexity in arithmetic.
Though $\mathsf{GLP}$ was extensively studied \cite{Ign92, Ign93, Bo93, BJV05, Shap08, Bek10a, BBI10, Pak12, BG13}, the question of finding an appropriate cut-free formulation for this logic seemingly remained open (Problem 3 from \cite{Prob12}).
In the present paper we introduce a proof system for $\mathsf{GLP}$ in the formalism of nested sequents\footnote{The idea to apply nested sequents in proof theory of $\mathsf{GLP}$ is due to Kai Br\"{u}nnler.} and prove the cut elimination theorem for it. 
The notion of nested sequent, invented several times independently (see \cite{Bu92, Ka94, Bru09, Pog10, Fit12}), naturally generalises both the notion of sequent (which is a nested
sequent of depth zero) and the notion of hypersequent (which is essentially a nested sequent of depth one). In brief, a nested sequent is a tree of ordinary sequents.

Many investigations in $\mathsf{GLP}$ employ a fragment of $\mathsf{GLP}$ denoted $\mathsf{J}$ \cite{Bek10a}. The reduction of $\mathsf{GLP}$ to $\mathsf{J}$ was first established in \cite{Bek10a} by involved model-theoretic arguments. Other proofs, on the basis of arithmetical semantics and topological semantics, were given in \cite{Bek11} and \cite{BG13}, respectively. In this note, we give a syntactic proof of the same reduction as an application of the cut elimination.

The plan of the paper is as follows:
in Section 1  we recall the Hilbert-style axiomatization of $\mathsf{GLP}$ and define the notion of nested sequent in the context of polymodal logic; in Section 2 we present the nested sequent formulation of $\mathsf{GLP}$ and obtain admissibility of basic structural rules; in Section 3 we prove the cut elimination theorem, which follows the corresponding proofs for the provability logic $\mathsf{GL}$ (see \cite{Bo83, Min05, Val83, GorRam12, Sas01}) and for systems of nested sequents \cite{Bru09}; in the final section we establish the reduction of $\mathsf{GLP}$ to its important fragment called $\mathsf{J}$ syntactically. 


\section{Preliminaries}
The polymodal provability logic $\mathsf{GLP}$ is a propositional modal logic in a language with infinitely many modalities $\Box_0$, $\Box_1$, $\Box_2$, etc. The dual connectives are denoted by $\Diamond_0$, $\Diamond_1$, $\Diamond_2$, etc.

\textit{Formulas} of $\mathsf{GLP}$, denoted by $A$, $B$, $C$, are built up as follows:
$$ A ::= p \,\,|\,\, \overline{p} \,\,|\,\, \large{\top} \,\,|\,\, \bot \,\,|\,\, (A \wedge A) \,\,|\,\,(A \vee A) \,\,|\,\, \Box_i A \,\,|\,\, \Diamond_i A \;, $$
where $p$ and $\overline{p}$ stand for atoms and their complements.

Let the \textit{complexity} $\lvert A \rvert$ of a formula $A$ be
\begin{align*}
& \lvert p \rvert = \lvert \overline{p} \rvert :=1,\\
& \lvert \top \rvert = \lvert \bot \rvert := 1,\\
& \lvert \Box_i A \rvert = \lvert \Diamond_i A \rvert := \lvert A \rvert +1,\\
& \lvert A \wedge B\rvert = \lvert A \vee B \rvert := \max \{\lvert A \rvert, \lvert B \rvert\} +1.
\end{align*}

The \textit{negation} $\overline{A}$ of a formula $A$ is defined in the usual way by De Morgan's laws, the law of double negation and the duality laws for the modal operators, i.e. we inductively define
\begin{gather*}
\overline{(p)} := \overline{p},\qquad
\overline{\overline{p}} := p,\\
\overline{\top} := \bot,\qquad
\overline{\bot} := \top,\\
\overline{(A \wedge B)} := (\overline{A} \vee \overline{B}),\qquad
\overline{(A \vee B)} := (\overline{A} \wedge \overline{B}),\\
\overline{\Box_i A} := \Diamond_i \overline{A},\qquad
\overline{\Diamond_i A} := \Box_i \overline{A}.
\end{gather*}

We also put
\begin{gather*}
A\rightarrow B := \overline{A} \vee B, \qquad A \leftrightarrow B := (A\rightarrow B) \wedge
(B\rightarrow A).
\end{gather*}

The Hilbert-style axiomatization of $\mathsf{GLP}$ is as follows:\\

\textit{Axioms:}
\begin{itemize}
\item[(i)] Boolean tautologies;
\item[(ii)] $\Box_i (A \rightarrow B) \rightarrow (\Box_i A \rightarrow \Box_i B)$;
\item[(iii)] $\Box_i ( \Box_i A \rightarrow A) \rightarrow \Box_i A$;
\item[(iv)] $\Diamond_i A \rightarrow \Box_j \Diamond_i A$ for $i < j$;
\item[(v)] $\Box_i A \rightarrow \Box_j A$ for $i \leqslant j$.
\end{itemize}

\textit{Rules:} modus ponens, $A / \Box_i A$. \\

In order to define a cut-free sequent system for $\mathsf{GLP}$ we adopt so-called nested (deep) sequents \cite{Bru09}.  
A \textit{nested sequent}, denoted by $\Gamma$, $\Delta$, $\Sigma$, $\Upsilon$, is inductively defined as a finite multiset of formulas and expressions of the form $[\Gamma]_i$, where $\Gamma$ is a nested sequent and $i$ is a natural number.
Nested sequents are often written without any curly braces, and the comma in the expression $\Gamma, \Delta$ means the multiset union. In the following nested sequents are referred merely as sequents.

For a sequent $\Gamma = A_1, \ldots ,A_n, [\Delta_1]_{i_1}, \ldots , [\Delta_m]_{i_m}$, its intended interpretation as a formula is
\begin{gather*}
\Gamma^\sharp :=
\begin{cases}
\bot & \text{if $n = m = 0$,}\\
 A_1 \vee \ldots \vee A_n \vee \Box_{i_1} \Delta^\sharp_1 \vee \ldots \vee \Box_{i_m} \Delta^\sharp_m & \text{otherwise.}
\end{cases}
\end{gather*}

Every sequent $\Gamma$ has a \textit{corresponding tree}, denoted \textit{tree}($\Gamma$), whose edges are marked with natural numbers and nodes are marked with multisets of formulas. The corresponding tree of the above sequent is  
\begin{center}
\begin{tikzpicture}
\tikzset{level distance=40pt, sibling distance=30pt, edge from parent/.style={-latex}}
\Tree [.\text{$A_1, \ldots , A_n $} \edge[draw] node[above, near end]{$i_1$}; \text{\emph{tree}($\Delta_1$)} \edge[draw] node[midway, below]{$i_2$}; \text{\emph{tree}($\Delta_2$)} \edge[-] node{}; \text{$\ldots$}  \edge[draw] node[midway, below]{$i_{m-1}$}; \text{\emph{tree}($\Delta_{m-1}$)} \edge[draw] node[near end, above]{$i_m$}; \text{\emph{tree}($\Delta_m$)} ] 
\end{tikzpicture}
\end{center}


A \textit{unary context} is defined as a sequent with a hole $\{ \:\}$, taking the place of a formula. Unary contexts will be denoted by $\Gamma \{\: \} $, $\Delta \{ \:\}$, $\Sigma \{ \:\}$. Given a unary context $\Gamma \{ \: \}$ and a sequent $\Upsilon$, we can obtain the sequent $\Gamma \{ \Upsilon \}$ by filling the hole in $\Gamma \{ \:\}$ with $\Upsilon$. In the following, we also use the notion of a \textit{sequent context} with multiple holes, which is defined as a sequent with $n$ different holes such that each hole takes the place of a formula and occurs exactly once in a context. A context with $n$ holes is denoted by \[\Gamma \underbrace{\{ \: \}\ldots\{ \: \}}_{n \text{ times}} \;.\] 
Notions of a context and the operation of filling holes are accurately defined in \cite{Bru10}.  


\section{The Sequent Calculus}
Now we introduce the sequent-style proof system for the provability logic $\mathsf{GLP}$.
The cut elimination theorem will be proved in the next two sections.

The sequent calculus $\mathsf{GLP_{NS}}$ is defined by the following initial sequents and inference rules: 
\subparagraph*{Initial sequents:}
\begin{gather*}
\AXC{ $\Gamma \{ p, \overline{p} \}$}
\DisplayProof \qquad
\AXC{ $\Gamma \{ \top \}$}
\DisplayProof
\end{gather*}
\subparagraph*{Propositional rules:}
\begin{gather*}
\AXC{$\Gamma \{ A \}$}
\AXC{$\Gamma \{ B \}$}
\LeftLabel{$\mathsf{\wedge}$}
\BIC{$\Gamma \{ A \wedge B \}$}
\DisplayProof \qquad
\AXC{$\Gamma \{ A,B \}$}
\LeftLabel{$\mathsf{\vee}$}
\UIC{$\Gamma \{ A \vee B \}$}
\DisplayProof
\end{gather*}
\subparagraph*{Modal rules:}
\begin{gather*}
\AXC{$\Gamma \{ [A, \Diamond_i \overline{A}]_i \}$}
\LeftLabel{$\mathsf{\Box}$}
\UIC{$\Gamma \{ \Box_i A\}$}
\DisplayProof \quad 
\AXC{$\Gamma \{\Diamond_i A, [A, \Delta ]_j\}$}
\LeftLabel{$\mathsf{\Diamond}$}
\RightLabel{$(i\leqslant j)$}
\UIC{$\Gamma \{\Diamond_i A, [\Delta ]_j\}$}
\DisplayProof\\\\
\AXC{$\Gamma \{\Diamond_i A, [\Diamond_i A, \Delta ]_j\}$}
\LeftLabel{$\mathsf{tran}$}
\RightLabel{$(i\leqslant j)$}
\UIC{$\Gamma \{\Diamond_i A, [\Delta ]_j\}$}
\DisplayProof  \quad 
\AXC{$\Gamma \{\Diamond_i A, [\Diamond_i A, \Delta ]_j\}$}
\LeftLabel{$\mathsf{eucl}$}
\RightLabel{$(i < j)$}
\UIC{$\Gamma \{ [\Diamond_i A, \Delta ]_j\}$}
\DisplayProof
\end{gather*}
\begin{center}
\textbf{Fig. 1}\, System $\mathsf{GLP_{NS}}$
\end{center}

In these inference rules, explicitly displayed formulas in the premises are called \emph{introducing} or \emph{auxiliary formulas} and explicitly displayed formulas in the conclusions are called \emph{introduced} or \emph{principal formulas}. The \emph{principal position} of an inference with the conclusion $\Gamma$ and the principal formula $A$ is a sequent context $\Delta \{ \: \}$ such that $\Gamma = \Delta \{ A\}$.

Recall that a \emph{derivation} in a sequent calculus is a finite tree whose nodes are marked
by sequents that is constructed according to the rules of the sequent calculus.
A \emph{proof} is defined as a derivation, where all leaves are labelled with initial sequents. A sequent $\Gamma$ is \emph{provable} in a sequent calculus if there is a proof with the root marked by $\Gamma$.
\begin{lem}\label{GenAx}
For any formula $A$, we have $\mathsf{GLP_{NS}} \vdash \Gamma \{ A, \overline{A}\}$. 
\end{lem}
\begin{proof}
Standard induction on the structure of $A$.
\end{proof}
The \emph{height} $\lvert \pi \rvert$ of a proof $\pi$ is the length of the longest branch in $\pi$.
A proof only consisting of an initial sequent has height $0$.
An inference rule is called \textit{admissible} (for a given proof system) if, for every instance of the rule, the conclusion is provable whenever all premises are provable. Let the cut rule, which will be proved to be admissible for $\mathsf{GLP_{NS}}$, be
\[\AXC{$\Gamma \{A\}$}
\AXC{$\Gamma \{\overline{A}\}$} \LeftLabel{$\mathsf{cut}$}
\RightLabel{ .} 
\BIC{$\Gamma \{ \emptyset \}$} 
\DisplayProof 
\]
Define the principal position of the cut rule as $\Gamma \{ \: \}$.
\begin{lem} \label{aux}
The following rules
\[
\AXC{$\Gamma \{\emptyset \} $}
\UIC{$\Gamma \{A \} $} 
\DisplayProof \quad
\AXC{$\Gamma$}
\UIC{$[\Gamma]_i $} 
\DisplayProof \quad
\AXC{$\Gamma \{A \vee B \} $}
\UIC{$\Gamma \{A,B \} $} 
\DisplayProof \quad
\AXC{$\Gamma \{\bot \} $}
\UIC{$\Gamma \{\emptyset\} $} 
\DisplayProof \quad
\AXC{$\Gamma \{\Box_i A \} $}
\UIC{$\Gamma \{[A]_i \} $} 
\DisplayProof
\]
are admissible for $\mathsf{GLP_{NS}} + \mathsf{cut}$.
\end{lem}
\begin{proof}
Admissibility of the first two rules is established by induction on the heights of proofs. Cases of other three rules are immediately established using the cut rule, Lemma \ref{GenAx} and admissibility of the first rule.
\end{proof}

Now we state the connection between $\mathsf{GLP_{NS}}$ and $\mathsf{GLP}$: 
\begin{prop}
$\mathsf{GLP_{NS}} + \mathsf{cut} \vdash \Gamma \Longleftrightarrow \mathsf{GLP} \vdash  \Gamma^\sharp$.
\end{prop}
\begin{proof}
The left-to-right part is obvious. To prove the converse, assume there is a proof $\pi$ of $\Gamma^\sharp$ in $\mathsf{GLP}$. By  induction on $\lvert \pi \rvert$, we immediately obtain $\mathsf{GLP_{NS}} + \mathsf{cut} \vdash \Gamma^\sharp$. Applying admissibility of the last three rules of Lemma \ref{aux}, we get $\mathsf{GLP_{NS}} + \mathsf{cut} \vdash \Gamma$.
\end{proof}

\section{Admissible rules}
In the present section we obtain admissibility of auxiliary inference rules, which will be applied in the cut elimination.

Call a finite set of formulas $\mathcal{C}$ \textit{adequate} if it is closed under subformulas and negation. For an adequate set $\mathcal{C}$, by $\mathsf{cut}(\mathcal{C})$ we denote the corresponding rule with the side condition $A \in \mathcal{C}$.
\begin{lem}[Admissibility of structural rules]\label{admissible structural rules}
The rules of weakening, merge and monotonicity   
\[
\AXC{$\Gamma \{ \emptyset \}$} 
\LeftLabel{$\mathsf{weak}$} \UIC{$\Gamma \{ \Delta \}$} \DisplayProof \quad
\AXC{$\Gamma \{ [\Delta]_i , [\Sigma]_i\}$} \LeftLabel{$\mathsf{merge}$} 
\UIC{$\Gamma \{ [\Delta, \Sigma]_i \}$} 
\DisplayProof \quad 
\AXC{$\Gamma \{ [ \Delta]_i\}$} \LeftLabel{$\mathsf{mon}$} \RightLabel{$(i\leqslant j)$} \UIC{$\Gamma \{ [\Delta]_j \}$}
\DisplayProof\]
are admissible for $\mathsf{GLP_{NS}} + \mathsf{cut}(\mathcal{C}) $.
\end{lem}
\begin{proof}
Simple transformations of proofs.
\end{proof}
Recall that an inference rule is called \textit{invertible} (for a given proof system) if, for every instance of the rule, all premises are provable whenever the conclusion is provable. 

\begin{lem}[Invertibility]\label{invertability 1}
For $\mathsf{GLP_{NS}} + \mathsf{cut}(\mathcal{C})$, all rules of the system and the rule
\[\AXC{$\Gamma \{ \emptyset \}$}
\LeftLabel{$\bot$}
\UIC{$\Gamma \{ \bot \}$}
\DisplayProof\]
are invertible.
\end{lem}
\begin{proof}
Standard induction on the heights of proofs.
\end{proof}
We stress that the contraction rule is also admissible for $\mathsf{GLP_{NS}} + \mathsf{cut}(\mathcal{C})$, but we use this rule only in the following weak form:
\begin{lem}
The rule
\[\AXC{$\Gamma \{ p,p \}$}
\UIC{$\Gamma \{ p \}$}
\DisplayProof\]
is admissible for $\mathsf{GLP_{NS}}$.
\end{lem}
\begin{proof}
Simple induction on the height of a proof of $\Gamma \{ p,p \}$.
\end{proof}
In the corresponding tree of a sequent context $\Gamma \{ \:\} \{ \:\}^n$, consider a path of the form $$a_0 \leftarrow_{j_1} a_1 \leftarrow_{j_2} a_2 \leftarrow_{j_3} \dotso \leftarrow_{j_k} a_k \to_{j_{k+1}} a_{k+1} \rightarrow_{j_{k+2}} \dotso  \to_{j_{k+l}} a_{k+l} , $$ where edges are directed away from the root. A path of this form is called an \textit{$i$-path} if $i < j_1 , \dotso , j_k $ and $i \leqslant j_{k+1} , \dotso , j_{k+l}$. We define a \textit{strict $i$-path} by letting $k=0$ in the previous definition.

Define the rule $\Box\text{-}\mathsf{cut}$ as
\[  
\AXC{$\Gamma \{\Diamond_i \overline{A} \} \{ \Diamond_i \overline{A}\}^n$}
\AXC{$\Gamma \{\Box_i A \} \{ \emptyset\}^n$} 
\LeftLabel{$\Box\text{-}\mathsf{cut}$}
\RightLabel{ ,}
\BIC{$\Gamma \{\emptyset\} \{ \emptyset\}^n$}
\DisplayProof
 \]
where, in the corresponding tree of $\Gamma \{ \:\} \{ \:\}^n$, there are $i$-paths from the node of the first hole to each node of the other holes. Define the principal position of $\Box\text{-}\mathsf{cut}$ as $\Gamma \{ \:\} \{ \emptyset\}^n$.

In the proof of the cut elimination, we need trace occurrences of $\Diamond_i \overline{A}$ from premisses of the rule $\Box$ throughout formal proofs:
\[\AXC{$\Gamma \{[A, \Diamond_i \overline{A}]_i \}$} 
\LeftLabel{$\mathsf{\Box}$}
\RightLabel{ .}
\UIC{$\Gamma \{\Box_i A \}$}
\DisplayProof
\]
To facilitate this treatment, we use \textit{annotated formulas} of the form $\diamondtimes_i B$ where $B$ is a ordinary formula. We also consider \textit{annotated proofs} obtained by allowing annotated formulas in sequents and annotated variants of the rules $\mathsf{\Diamond}$ and $\mathsf{tran}$:
\begin{gather*}
\AXC{$\Gamma \{\diamondtimes_i A , [ A, \Delta ]_j\}$}
\LeftLabel{$\mathsf{\diamondtimes}$}
\RightLabel{$(i\leqslant j)$}
\UIC{$\Gamma \{\diamondtimes_i A, [ \Delta ]_j\}$}
\DisplayProof \:, \quad
\AXC{$\Gamma \{\diamondtimes_i A, [ \diamondtimes_i A, \Delta ]_j\}$}
\LeftLabel{$\mathsf{tran}^\prime$}
\RightLabel{$(i \leqslant j)$}
\UIC{$\Gamma \{ \diamondtimes_i A, [ \Delta ]_j\}$}
\DisplayProof \:.
\end{gather*}
Note that we don't allow the annotated variant for the rule $\mathsf{eucl}$. 

Let the rule $\boxplus\text{-}\mathsf{cut}$ be
\[  
\AXC{$\Gamma \{\diamondtimes_i \overline{A} \} \{ \diamondtimes_i \overline{A}\}^n$}
\AXC{$\Gamma \{\Box_i A \} \{ \emptyset\}^n$} 
\LeftLabel{$\boxplus\text{-}\mathsf{cut}$}
\RightLabel{ ,}
\BIC{$\Gamma \{\emptyset\} \{ \emptyset\}^n$}
\DisplayProof
 \]
where, in the corresponding tree of $\Gamma \{ \:\} \{ \:\}^n$, there are strict $i$-paths from the node of the first hole to the every node of the others. Define the principal position of $\boxplus\text{-}\mathsf{cut}$ as $\Gamma\{ \:\} \{ \emptyset\}^n$.





For an adequate set $\mathcal{C}$, by $\Box\text{-}\mathsf{cut}(\mathcal{C})$ and $\boxplus\text{-}\mathsf{cut}(\mathcal{C})$ we denote the corresponding rules with the side condition $A \in \mathcal{C}$.

Let us define the rule
\[\AXC{$\Gamma \{[\Delta]_i \} \{ \emptyset\} $} 
\LeftLabel{$\mathsf{str}$}
\UIC{$\Gamma \{ \emptyset \} \{[\Delta]_i \}$}
\DisplayProof\]
with the proviso that there is an $i$-path from the node of the first hole to the node of the second hole in the corresponding tree of $\Gamma \{ \:\} \{ \:\}$.
\begin{lem}\label{stretch}
The rule $\mathsf{str}$ is admissible for $\mathsf{GLP_{NS}} + \mathsf{cut}(\mathcal{C}) $.
\end{lem} 
\begin{proof}
The rule $\mathsf{str}$ moves a boxed sequent $[\Delta]_i$ inside a sequent from one place to another. The one-step moving rules
\begin{gather*}
\AXC{$\Gamma \{ [\Delta]_i, [\Sigma ]_j\}$}
\RightLabel{$(i \leqslant j)$}
\UIC{$\Gamma \{ [[\Delta ]_i, \Sigma ]_j\}$}
\DisplayProof \quad 
\AXC{$\Gamma \{ [[\Delta]_i, \Sigma ]_j\}$}
\RightLabel{$(i< j)$}
\UIC{$\Gamma \{[\Delta]_i, [\Sigma ]_j\}$}
\DisplayProof 
\end{gather*}
are admissible from definitions of rules $\mathsf{tran}$ and $\mathsf{eucl}$. Hence, the rule $\mathsf{str}$ is admissible.
\end{proof}

Let us define the rule
\[\AXC{$\Gamma \{[\Delta]_i \} \{ \emptyset\} $} 
\LeftLabel{$\mathsf{upstr}$}
\UIC{$\Gamma \{ \emptyset \} \{[\Delta]_i \}$}
\DisplayProof\]
with the proviso that there is a strict $i$-path from the node of the first hole to the node of the second hole in the corresponding tree of $\Gamma \{ \:\} \{ \:\}$.

In the corresponding tree of a sequent $\Gamma$, the \emph{depth} of a node is the length of the path from the node to the root. The depth of the root is $0$.
By $\boxplus\text{-}\mathsf{cut}_d(\mathcal{C})$ we denote the rule $\boxplus\text{-}\mathsf{cut}(\mathcal{C})$
with the requirement that the depth the node of the hole $\{ \: \}$ in the principal position of $\boxplus\text{-}\mathsf{cut}(\mathcal{C})$ is greater or equal than $d$.

\begin{lem}\label{upstretch}
In $\mathsf{GLP_{NS}} + \mathsf{cut}(\mathcal{C}) + \boxplus\text{-}\mathsf{cut}_d(\mathcal{C}) $, if there is an annotated proof for the premise of $\mathsf{upstr}$, then there is an annotated proof for the conclusion of $\mathsf{upstr}$.
\end{lem} 
\begin{proof}
The rule $\mathsf{upstr}$ moves a boxed sequent $[\Delta]_i$ inside a sequent from one place to deeper position. The one-step moving rule
\begin{gather*}
\AXC{$\Gamma \{ [\Delta]_i, [\Sigma ]_j\}$}
\LeftLabel{$\rho$}
\RightLabel{$(i \leqslant j)$}
\UIC{$\Gamma \{ [[\Delta ]_i, \Sigma ]_j\}$}
\DisplayProof 
\end{gather*}
is admissible for $\mathsf{GLP_{NS}} + \mathsf{cut}(\mathcal{C}) + \boxplus\text{-}\mathsf{cut}(\mathcal{C})$ from the definition of rules $\mathsf{tran}$ and $\mathsf{eucl}$. Moreover, all applications of the rule $\boxplus\text{-}\mathsf{cut}(\mathcal{C})$ in the proof of the conclusion of $\rho$ can become only deeper. Hence, the one-step moving rule and the rule $\mathsf{upstr}$ are also admissible for $\mathsf{GLP_{NS}} + \mathsf{cut}(\mathcal{C}) + \boxplus\text{-}\mathsf{cut}_d(\mathcal{C}) $.
\end{proof}

\begin{lem} \label{ann lemma}
For $\mathsf{GLP_{NS}} + \mathsf{cut}(\mathcal{C})  +\boxplus\text{-}\mathsf{cut}_d(\mathcal{C})$, the rule  
\[
\AXC{$\Gamma \{ [\Diamond_i A , \Delta]_i \}$}
\UIC{$\Gamma \{ [\diamondtimes_i A , \Delta]_i \}$}
\DisplayProof
\]
is admissible with respect to annotated proofs.
\end{lem} 
\begin{proof}
In an annotated proof of $\Gamma \{ [\Diamond_i A , \Delta]_i \}$, consider all applications of the rule $\mathsf{eucl}$ with the principal formula being an ancestor of $\Diamond_i A$. All applications of these kind are redundant. Thus, we can obtain a proof $\pi$ of $\Gamma \{ [\Diamond_i A , \Delta]_i \}$ without these applications of $\mathsf{eucl}$. We annotate all ancestors of $\Diamond_i A$ in $\pi$ and obtain the annotated proof of $\Gamma \{ [\diamondtimes_i A , \Delta]_i \}$.
\end{proof}

\begin{lem}[Invertibility]\label{invertability}
For $\mathsf{GLP_{NS}} + \mathsf{cut}(\mathcal{C}) +\boxplus\text{-}\mathsf{cut}_d(\mathcal{C})$, all rules of the system and the rule
\[\AXC{$\Gamma \{ \emptyset \}$}
\LeftLabel{$\bot$}
\UIC{$\Gamma \{ \bot \}$}
\DisplayProof\]
are invertible with respect to annotated proofs.
\end{lem}

\begin{lem}[Admissibility of structural rules]\label{admissible structural rules}
The rules of weakening, merge and monotonicity   
\[
\AXC{$\Gamma \{ \emptyset \}$} 
\LeftLabel{$\mathsf{weak}$} \UIC{$\Gamma \{ \Delta \}$} \DisplayProof \quad
\AXC{$\Gamma \{ [\Delta]_i , [\Sigma]_i\}$} \LeftLabel{$\mathsf{merge}$} 
\UIC{$\Gamma \{ [\Delta, \Sigma]_i \}$} 
\DisplayProof \quad 
\AXC{$\Gamma \{ [ \Delta]_i\}$} \LeftLabel{$\mathsf{mon}$} \RightLabel{$(i\leqslant j)$} \UIC{$\Gamma \{ [\Delta]_j \}$}
\DisplayProof\]
are admissible for $\mathsf{GLP_{NS}} + \mathsf{cut}(\mathcal{C}) +\boxplus\text{-}\mathsf{cut}_d(\mathcal{C})$ with respect to annotated proofs.
\end{lem}
\begin{proof}
Simple transformations of proofs.
\end{proof}

\section{Cut Elimination}
In the present section we prove admissibility of the cut rule for $\mathsf{GLP_{NS}}$.
\begin{lem} \label{modal cut}
For an inference 
\[  
\AXC{$\pi_1$}
\noLine
\UIC{\vdots}
\noLine
\UIC{$\Gamma \{\Diamond_i \overline{A} \} \{ \Diamond_i \overline{A}\}^n$}
\AXC{$\pi_2$}
\noLine
\UIC{\vdots}
\noLine
\UIC{$\Gamma \{\Box_i A \} \{ \emptyset\}^n$} 
\LeftLabel{$\Box\text{-}\mathsf{cut}(\mathcal{C})$}
\RightLabel{$(A \in \mathcal{C})$}
\BIC{$\Gamma \{\emptyset\} \{ \emptyset\}^n$}
\DisplayProof
 \] 
where $\pi_1$ and $\pi_2$ are ordinary proofs in $\mathsf{GLP_{NS}} + \mathsf{cut}(\mathcal{C})$, there is an annotated proof of $\Gamma \{\emptyset\} \{ \emptyset\}^n $ in $\mathsf{GLP_{NS}} + \mathsf{cut}(\mathcal{C}) +\boxplus\text{-}\mathsf{cut}(\mathcal{C})$.

\end{lem}

\begin{proof}
We prove $\mathsf{GLP_{NS}} + \mathsf{cut}(\mathcal{C})  +\boxplus\text{-}\mathsf{cut}(\mathcal{C}) \vdash \Gamma \{\emptyset\} \{ \emptyset\}^n$ by induction on $\rvert \pi_1 \lvert $. 
If $\rvert \pi_1 \lvert  = 0$, then $\Gamma \{\Diamond_i \overline{A} \} \{ \Diamond_i \overline{A}\}^n$ is an initial sequent. Hence, $\Gamma \{\emptyset\} \{ \emptyset\}^n$ is an initial sequent and $\mathsf{GLP_{NS}} + \mathsf{cut}(\mathcal{C}) +\boxplus\text{-}\mathsf{cut}(\mathcal{C}) \vdash \Gamma \{\emptyset\} \{ \emptyset\}^n$. 
Otherwise, consider the lowermost application of an inference rule in $\pi_1$. The proof $\pi_1$ has one of the following forms:
\[\AXC{$\pi^\prime_1$}
\noLine
\UIC{\vdots}
\noLine
\UIC{$\Gamma^\prime \{\Diamond_i \overline{A} \} \{ \Diamond_i \overline{A}\}^n$}
\LeftLabel{$\rho$}
\UIC{$\Gamma \{\Diamond_i \overline{A} \} \{ \Diamond_i \overline{A}\}^n$}
\DisplayProof \qquad
\AXC{$\pi^\prime_1$}
\noLine
\UIC{\vdots}
\noLine
\UIC{$\Gamma^\prime \{\Diamond_i \overline{A} \} \{ \Diamond_i \overline{A}\}^n$}
\AXC{$\pi^{\prime \prime}_1$}
\noLine
\UIC{\vdots}
\noLine
\UIC{$\Gamma^{\prime \prime} \{\Diamond_i \overline{A} \} \{ \Diamond_i \overline{A}\}^n$}
\LeftLabel{$\rho$}
\RightLabel{ .}
\BIC{$\Gamma \{\Diamond_i \overline{A} \} \{ \Diamond_i \overline{A}\}^n$}
\DisplayProof\]

Case 1. Suppose the principal position of this lowermost inference in $\pi_1$ coincides with one of the holes in $\Gamma \{\: \} \{ \:\}^n$. Then the rule $\rho$ equals to $\mathsf{tran}$, $\mathsf{eucl}$ or $\Diamond$.


Subcase A: the rule $\rho$ equals to $\mathsf{tran}$ or $\mathsf{eucl}$. The lowermost rule application in $\pi_1$ has the form
\begin{gather*}
\AXC{$\pi^\prime_1$}
\noLine
\UIC{\vdots}
\noLine
\UIC{$\Delta \{\Diamond_i \overline{A} \} \{ \Diamond_i \overline{A}\}^n \{ \Diamond_i \overline{A}\}$}
\LeftLabel{$\rho$}
\RightLabel{ ,}
\UIC{$\Delta \{\Diamond_i \overline{A} \} \{ \Diamond_i \overline{A}\}^n \{ \emptyset\}$}
\DisplayProof \label{Diamond form}
\end{gather*}
where $\Delta \{ \: \} \{ \: \}^n \{ \emptyset \} = \Gamma \{ \: \} \{ \: \}^n $, $\Delta \{ \: \} \{ \: \}^n \{ \Diamond_i \overline{A} \} = \Gamma^\prime \{ \: \} \{ \: \}^n $ and, in the sequent context $\Delta \{ \: \} \{ \: \}^n \{ \: \}$, there is an $i$-path from the node of the first hole to the node of the last hole. We see
\[
\AXC{$\pi^\prime_1$}
\noLine
\UIC{\vdots}
\noLine
\UIC{$\Delta \{\Diamond_i \overline{A} \} \{ \Diamond_i \overline{A}\}^n \{\Diamond_i \overline{A} \}$}
\AXC{$\pi_2$}
\noLine
\UIC{\vdots}
\noLine
\UIC{$\Delta \{\Box_i A_i \} \{ \emptyset\}^n \{ \emptyset\}$}
\LeftLabel{$\Box\text{-}\mathsf{cut}(\mathcal{C})$}
\RightLabel{$( A \in \mathcal{C})$,}
\BIC{$\Delta \{\emptyset\} \{ \emptyset\}^n \{ \emptyset\}$}
\DisplayProof 
 \]
where $\Delta \{ \emptyset \} \{ \emptyset \}^n \{ \emptyset \} = \Gamma \{ \emptyset \} \{ \emptyset \}^n $, $\Delta \{  \Diamond_i \overline{A} \} \{  \Diamond_i \overline{A}\}^n \{ \Diamond_i \overline{A} \} = \Gamma^\prime \{  \Diamond_i \overline{A} \} \{  \Diamond_i \overline{A} \}^n $ and $\Delta \{ \Box_i A]_i \} \{ \emptyset \}^n \{ \emptyset \} = \Gamma \{ \Box _i A \} \{ \emptyset \}^n $. 
Applying the induction hypothesis for $\pi^\prime_1$, we obtain $\mathsf{GLP_{NS}} + \mathsf{cut}(\mathcal{C}) +\boxplus\text{-}\mathsf{cut}(\mathcal{C}) \vdash \Gamma \{\emptyset\} \{ \emptyset\}^n$. 

Subcase B: $\rho$ equals to $\Diamond$.
The lowermost rule application in $\pi_1$ has the form
\begin{gather*}
\AXC{$\pi^\prime_1$}
\noLine
\UIC{\vdots}
\noLine
\UIC{$\Delta \{ \Diamond_i \overline{A} , [\overline{A}, \Sigma \{ \Diamond_i \overline{A} \}^k]_j\} \{ \Diamond_i \overline{A} \}^l$}
\LeftLabel{$\Diamond$}
\RightLabel{$(i\leqslant j)$,}
\UIC{$\Delta \{ \Diamond_i \overline{A} , [\Sigma \{ \Diamond_i \overline{A} \}^k]_j\} \{ \Diamond_i \overline{A} \}^l$}
\DisplayProof \label{Diamond form}
\end{gather*}
where sequent contexts $\Delta \{ \{ \: \}, [ \Sigma \{ \: \}^k]_j \} \{ \: \}^l $ and $\Delta \{ \{ \: \}, [\overline{A}, \Sigma \{ \: \}^k]_j \} \{ \: \}^l$ coincide with $\Gamma \{ \: \} \{ \: \}^n $ and $\Gamma^\prime \{ \: \} \{ \: \}^n $ up to a permutation of holes, respectively. We see
\[
\AXC{$\pi^\prime_1$}
\noLine
\UIC{\vdots}
\noLine
\UIC{$\Gamma^\prime \{\Diamond_i \overline{A} \} \{ \Diamond_i \overline{A}\}^n$}
\AXC{$\pi_2$}
\noLine
\UIC{\vdots}
\noLine
\UIC{$\Gamma \{\Box _i A \} \{ \emptyset\}^n$}
\LeftLabel{$\mathsf{weak}$}
\UIC{$\Gamma^\prime \{\Box_i A \} \{ \emptyset\}^n$}
\LeftLabel{$\Box\text{-}\mathsf{cut}(\mathcal{C})$}
\RightLabel{$( A \in \mathcal{C})$,}
\BIC{$\Gamma^\prime \{\emptyset\} \{ \emptyset\}^n$}
\DisplayProof 
 \]
where the rule $\mathsf{weak}$ is admissible for $\mathsf{GLP_{NS}} + \mathsf{cut}(\mathcal{C}) $ by Lemma \ref{admissible structural rules} and $\Gamma^\prime \{\emptyset\} \{ \emptyset\}^n = \Delta \{  [\overline{A}, \Sigma \{ \emptyset \}^k]_j \} \{ \emptyset \}^l$. Applying the induction hypothesis for $\pi^\prime_1$, we get $\mathsf{GLP_{NS}} + \mathsf{cut}(\mathcal{C}) + \boxplus\text{-}\mathsf{cut}(\mathcal{C}) \vdash \Delta \{  [\overline{A}, \Sigma \{ \emptyset \}^k]_j \} \{ \emptyset \}^l$. 

Now we claim $\mathsf{GLP_{NS}} + \mathsf{cut}(\mathcal{C}) + \boxplus\text{-}\mathsf{cut}(\mathcal{C}) \vdash \Delta \{  [A, \Sigma \{ \emptyset \}^k]_j \} \{ \emptyset \}^l$.  
By Lemma \ref{invertability}, the rule $\Box$ is invertible for $\mathsf{GLP_{NS}} + \mathsf{cut}(\mathcal{C}) $. We see
\begin{gather}\label{deriv1}
\AXC{$\pi_2$}
\noLine
\UIC{\vdots}
\noLine
\UIC{$\Gamma \{\Box_i A \} \{ \emptyset\}^n$}
\LeftLabel{$\mathsf{\delta}$}
\UIC{$\Gamma \{[A, \Diamond_i \overline{A}]_i, \} \{ \emptyset\}^n$}
\LeftLabel{$\mathsf{str}$}
\RightLabel{ ,}
\UIC{$\Delta \{ [A, \Diamond_i \overline{A}]_i, [ \Sigma \{ \emptyset \}^k]_j \} \{ \emptyset \}^l$}
\DisplayProof
\end{gather}
where $\delta$ is the inverse of the rule $\Box$. We obtain $\mathsf{GLP_{NS}} + \mathsf{cut}(\mathcal{C}) \vdash \Delta \{ [A, \Diamond_i \overline{A}]_i, [ \Sigma \{ \emptyset \}^k]_j \} \{ \emptyset \}^l$ and $\mathsf{GLP_{NS}} + \mathsf{cut}(\mathcal{C})  +\boxplus\text{-}\mathsf{cut}(\mathcal{C}) \vdash \Delta \{ [A, \Diamond_i \overline{A}]_i, [ \Sigma \{ \emptyset \}^k]_j \} \{ \emptyset \}^l$. By Lemma \ref{ann lemma}, there is an annotated proof of $\Delta \{ [A, \diamondtimes_i \overline{A}]_i, [ \Sigma \{ \emptyset \}^k]_j \} \{ \emptyset \}^l$ in $\mathsf{GLP_{NS}} + \mathsf{cut}(\mathcal{C})  +\boxplus\text{-}\mathsf{cut}(\mathcal{C})$. 

Continuing the derivation of (\ref{deriv1}), we see
\begin{gather*}
\AXC{$\Delta \{ [A, \diamondtimes_i \overline{A}]_i, [ \Sigma \{ \emptyset \}^k]_j \} \{ \emptyset \}^l$}
\LeftLabel{$\mathsf{mon}$}
\RightLabel{$(i\leqslant j)$}
\UIC{$\Delta \{ [A, \diamondtimes_i \overline{A}]_j, [ \Sigma \{ \emptyset \}^k]_j \} \{ \emptyset \}^l$}
\LeftLabel{$\mathsf{merge}$}
\UIC{$\Delta \{ [A, \diamondtimes_i \overline{A}, \Sigma \{ \emptyset \}^k]_j \} \{ \emptyset \}^l$}
\AXC{$\Delta \{ [A, \Diamond_i \overline{A}]_i, [ \Sigma \{ \emptyset \}^k]_j \} \{ \emptyset \}^l$}
\LeftLabel{$\mathsf{upstr}$}
\UIC{$\Delta \{ [[A, \Diamond_i \overline{A}]_i,  \Sigma \{ \emptyset \}^k]_j \} \{ \emptyset \}^l$}
\LeftLabel{$\mathsf{\Box}$}
\UIC{$\Delta \{ [\Box_i A,  \Sigma \{ \emptyset \}^k]_j \} \{ \emptyset \}^l$}
\LeftLabel{$\mathsf{weak}$}
\UIC{$\Delta \{ [A, \Box_i A,  \Sigma \{ \emptyset \}^k]_j \} \{ \emptyset \}^l$}
\LeftLabel{$\boxplus\text{-}\mathsf{cut}(\mathcal{C})$}
\RightLabel{$(  A \in \mathcal{C})$,}
\BIC{$\Delta \{ [ A,  \Sigma \{ \emptyset \}^k]_j \} \{ \emptyset \}^l$}
\DisplayProof
\end{gather*}
where rules $\mathsf{mon}$,  $\mathsf{merge}$, $\mathsf{weak}$ and $\mathsf{upstr}$ are admissible for $\mathsf{GLP_{NS}} + \mathsf{cut}(\mathcal{C})  +\boxplus\text{-}\mathsf{cut}(\mathcal{C})$ by Lemmata \ref{admissible structural rules} and \ref{upstretch}.

We obtain that sequents $\Delta \{  [\overline{A}, \Sigma \{ \emptyset \}^k]_j \} \{ \emptyset \}^l$ and $\Delta \{  [A, \Sigma \{ \emptyset \}^k]_j \} \{ \emptyset \}^l$ are provable in $\mathsf{GLP_{NS}} + \mathsf{cut}(\mathcal{C})  +\boxplus\text{-}\mathsf{cut}(\mathcal{C})$.
Applying the rule $\mathsf{cut}(\mathcal{C})$, we get $\mathsf{GLP_{NS}} + \mathsf{cut}(\mathcal{C})  +\boxplus\text{-}\mathsf{cut}(\mathcal{C}) \vdash \Delta \{  [ \Sigma \{ \emptyset \}^k]_j \} \{ \emptyset \}^l$. 
Recall that $\Gamma \{\emptyset\} \{ \emptyset\}^n = \Delta \{  [\Sigma \{ \emptyset \}^k]_j \} \{ \emptyset \}^l$. Then we see $\mathsf{GLP_{NS}} + \mathsf{cut}(\mathcal{C}) + \boxplus\text{-}\mathsf{cut}(\mathcal{C}) \vdash \Gamma \{\emptyset\} \{ \emptyset\}^n$.

Case 2.
Suppose the principal position of the lowermost rule application $\rho$ in $\pi_1$ differs with every hole in $\Gamma \{\: \} \{ \:\}^n$.
By Lemma \ref{invertability}, the rule $\rho$ is invertible for $\mathsf{GLP_{NS}} + \mathsf{cut}(\mathcal{C})$. Thus, we have
\[
\AXC{$\pi^\prime_1$}
\noLine
\UIC{\vdots}
\noLine
\UIC{$\Gamma^\prime \{\Diamond_i \overline{A} \} \{ \Diamond_i \overline{A}\}^n$}
\AXC{$\pi_2$}
\noLine
\UIC{\vdots}
\noLine
\UIC{$\Gamma \{\Box_i A \} \{ \emptyset\}^n$}
\LeftLabel{$\overline{\rho}$}
\UIC{$\Gamma^\prime \{\Box_i A \} \{ \emptyset\}^n$}
\LeftLabel{$\Box\text{-}\mathsf{cut}(\mathcal{C})$}
\RightLabel{$(  A \in \mathcal{C})$,}
\BIC{$\Gamma^\prime \{\emptyset\} \{ \emptyset\}^n$}
\DisplayProof 
 \]
where $\overline{\rho}$ is the corresponding inverse of $\rho$. 
Applying the induction hypothesis for $\pi^\prime_1$, we have $\mathsf{GLP_{NS}} + \mathsf{cut}(\mathcal{C}) + \boxplus\text{-}\mathsf{cut}(\mathcal{C}) \vdash \Gamma^\prime \{\emptyset\} \{ \emptyset\}^n$. If the rule $\rho$ has two premises, then we have $\mathsf{GLP_{NS}} + \mathsf{cut}(\mathcal{C}) + \boxplus\text{-}\mathsf{cut}(\mathcal{C}) \vdash \Gamma^{\prime \prime} \{\emptyset\} \{ \emptyset\}^n$ analogously.
Applying the rule $\rho$ to the sequent $\Gamma^\prime \{\emptyset\} \{ \emptyset\}^n$ (to the sequents $\Gamma^\prime \{\emptyset\} \{ \emptyset\}^n$ and $\Gamma^{\prime \prime} \{\emptyset\} \{ \emptyset\}^n$),
we immediately obtain $\mathsf{GLP_{NS}} + \mathsf{cut}(\mathcal{C}) + \boxplus\text{-}\mathsf{cut}(\mathcal{C}) \vdash \Gamma \{\emptyset\} \{ \emptyset\}^n$.

 

\end{proof}


Let us denote the adequate set of all proper subformulas of a formula $A$ and their negations by $\mathcal{C}_A$.
 
\begin{lem} \label{reduction lemma}
For an inference 
\[  
\AXC{$\pi_1$}
\noLine
\UIC{\vdots}
\noLine
\UIC{$\Gamma \{A\}$}
\AXC{$\pi_2$}
\noLine
\UIC{\vdots}
\noLine
\UIC{$\Gamma \{\overline{A}\}$} \LeftLabel{$\mathsf{cut}(A)$} 
\RightLabel{ ,}
\BIC{$\Gamma \{ \emptyset \}$} 
\DisplayProof
 \] 
where $\pi_1$ and $\pi_2$ are ordinary proofs in $\mathsf{GLP_{NS}} + \mathsf{cut}(\mathcal{C}_A)$, there is an annotated proof of $\Gamma \{\emptyset\} $ in $\mathsf{GLP_{NS}} + \mathsf{cut}(\mathcal{C}_A) + \boxplus\text{-}\mathsf{cut}(\mathcal{C}_A)$.
\end{lem} 
\begin{proof}
We prove $\mathsf{GLP_{NS}} + \mathsf{cut}(\mathcal{C}_A) + \boxplus\text{-}\mathsf{cut}(\mathcal{C}_A)  \vdash \Gamma \{ \emptyset \}$ by induction on the structure of the cut formula $A$.


Case 1: $A$ is of the form $p$ (or $\overline{p} $). The case is established by standard sub-induction on $\rvert \pi_1 \lvert$.
Both cases of $p$ or $\overline{p} $ are completely analogous. Hence, we assume $A = p$. If $\rvert \pi_1 \lvert =0$, then $\Gamma \{p\}$ is an initial sequent. Suppose $\Gamma \{\emptyset\}$ is also an initial sequent. Then we immediately have $\mathsf{GLP_{NS}} + \mathsf{cut}(\mathcal{C}_A) + \boxplus\text{-}\mathsf{cut}(\mathcal{C}_A)  \vdash \Gamma \{ \emptyset \}$. Otherwise, $\Gamma \{ \: \}$ has the form $\Delta \{ \overline{p}, \{ \: \} \}$. Then $\pi_2$ is a proof of $\Delta \{ \overline{p},  \overline{p}  \}$. Applying admissibility of the contraction rule for $\mathsf{GLP_{NS}} + \mathsf{cut}(\mathcal{C}_A) $ (see Lemma \ref{admissible structural rules}), we obtain  
$\mathsf{GLP_{NS}} + \mathsf{cut}(\mathcal{C}_A)  \vdash \Delta \{ \overline{p}\}$ and $\mathsf{GLP_{NS}} + \mathsf{cut}(\mathcal{C}_A) + \boxplus\text{-}\mathsf{cut}(\mathcal{C}_A)  \vdash \Delta \{ \overline{p}\}$. Recall that $ \Delta \{ \overline{p}\} = \Gamma \{ \emptyset \} $. The induction step is straightforward, so we omit it.

Case 2: $A$ is of the form $\top$ (or $\bot$). W.l.o.g. we assume $A = \top$. Then we have
\[\AXC{$\pi_2$}
\noLine
\UIC{\vdots}
\noLine
\UIC{$\Gamma \{ \bot\}$}
\LeftLabel{$\rho$}
\RightLabel{ ,}
\UIC{$\Gamma \{ \emptyset \}$} 
\DisplayProof
\]
where the rule $\rho$ is admissible for $\mathsf{GLP_{NS}} + \mathsf{cut}(\mathcal{C}_A) $ by Lemma \ref{invertability}.

Case 3: $A$ has the form $B \wedge C$ (or $\overline{B} \vee \overline{C}$). W.l.o.g. we assume $A= B \wedge C$. By Lemma \ref{invertability}, the introduction rules for $\mathsf{\wedge}$ and $\mathsf{\vee}$ are invertible for $\mathsf{GLP_{NS}} + \mathsf{cut}(\mathcal{C}_A) $.
Then we have
\begin{gather*}
\AXC{$\pi_1$}
\noLine
\UIC{\vdots}
\noLine
\UIC{$\Gamma \{ B \wedge C\}$}
\LeftLabel{$\mu_1$}
\UIC{$\Gamma \{ B \} $} 
\AXC{$\pi_1$}
\noLine
\UIC{\vdots}
\noLine
\UIC{$\Gamma \{ B \wedge C\}$}
\LeftLabel{$\mu_2$}
\UIC{$\Gamma \{ C \} $}
\LeftLabel{$\mathsf{weak}$}
\UIC{$\Gamma \{ \overline{B},C \} $}
\AXC{$\pi_2$}
\noLine
\UIC{\vdots}
\noLine
\UIC{$\Gamma \{ \overline{B} \vee \overline{C}\}$}
\LeftLabel{$\mu_3$}
\UIC{$\Gamma \{ \overline{B},\overline{C} \} $} 
\LeftLabel{$\mathsf{cut}(\mathcal{C}_A)$}
\RightLabel{$(C \in \mathcal{C}_A)$}
\BIC{$\Gamma \{ \overline{B} \} $} 
\LeftLabel{$\mathsf{cut}(\mathcal{C}_A)$}
\RightLabel{$(B \in \mathcal{C}_A)$,}
\BIC{$\Gamma \{ \emptyset \}$}
\DisplayProof
\end{gather*}
where $\mu_1$, $\mu_2$ and $\mu_3$ are the inverses of the introduction rules for $\mathsf{\wedge}$ and $\mathsf{\vee}$, and the rule $\mathsf{weak}$ is admissible by Lemma \ref{admissible structural rules}. 

Case 4: $A$ is of the form $\Box_i B$ (or $\Diamond_i \overline{B}$). W.l.lo.g. we assume $A= \Box_i B$. We have
\[\AXC{$\pi_1$}
\noLine
\UIC{\vdots}
\noLine
\UIC{$\Gamma \{\Box_i B\}$}
\AXC{$\pi_2$}
\noLine
\UIC{\vdots}
\noLine
\UIC{$\Gamma \{\Diamond_i \overline{B}\}$}
\LeftLabel{$\Box\text{-}\mathsf{cut}(\mathcal{C}_A)$} 
\RightLabel{$( B \in \mathcal{C}_A)$.}
\BIC{$\Gamma \{ \emptyset \}$} 
\DisplayProof
\]
By Lemma \ref{modal cut} there is an annotated proof of $\Gamma \{\emptyset\} $ in $\mathsf{GLP_{NS}} + \mathsf{cut}(\mathcal{C}_A) + \boxplus\text{-}\mathsf{cut}(\mathcal{C}_A)$.
\end{proof}

\begin{lem} \label{upmodal cut} \label{sasaki lemma 1}
For an inference 
\[  
\AXC{$\pi_1$}
\noLine
\UIC{\vdots}
\noLine
\UIC{$\Gamma \{\diamondtimes_i \overline{A} \} \{ \diamondtimes_i \overline{A}\}^n$}
\AXC{$\pi_2$}
\noLine
\UIC{\vdots}
\noLine
\UIC{$\Gamma \{\Box_i A \} \{ \emptyset\}^n$} 
\LeftLabel{$\boxplus\text{-}\mathsf{cut}_d(\mathcal{C})$}
\RightLabel{$( A \in \mathcal{C})$,}
\BIC{$\Gamma \{\emptyset\} \{ \emptyset\}^n$}
\DisplayProof
 \] 
where $\pi_1$ and $\pi_2$ are annotated proofs in $\mathsf{GLP_{NS}} + \mathsf{cut}(\mathcal{C}) + \boxplus\text{-}\mathsf{cut}_{d+1}(\mathcal{C}) $, there is an annotated proof of $\Gamma \{\emptyset\} \{ \emptyset\}^n $ in $\mathsf{GLP_{NS}} + \mathsf{cut}(\mathcal{C}) + \boxplus\text{-}\mathsf{cut}_{d+1}(\mathcal{C}) $.

\end{lem} 
\begin{proof}
The proof is analogous to the proof of Lemma \ref{modal cut}. The only difference is as follows: The principal position of an application the rule $\mathsf{eucl}$ can't coincide with one of the holes in $\Gamma \{\:\} \{ \:\}^n$. Thus, all applications of the rule $\mathsf{str}$ from the proof of Lemma \ref{modal cut} appear to be the applications of the rule $\mathsf{upstr}$. For these applications we apply Lemma \ref{upstretch} instead of Lemma \ref{stretch}. In addition, the applications of the rule $\boxplus\text{-}\mathsf{cut}(\mathcal{C}) $ appear to be applications of $\boxplus\text{-}\mathsf{cut}_{d+1}(\mathcal{C}) $. 
\end{proof}

\begin{lem} \label{sasaki lemma 2}
For any sequent $\Gamma$ and any adequate set $\mathcal{C}$ there is a natural number $d$ such that if there is an annotated proof of $\Gamma$ in $\mathsf{GLP_{NS}} + \mathsf{cut}(\mathcal{C}) + \boxplus\text{-}\mathsf{cut}_d(\mathcal{C})$, then there is an ordinary proof of $\Gamma$ in $\mathsf{GLP_{NS}} + \mathsf{cut}(\mathcal{C})$.
\end{lem}

\begin{proof}



For the given sequent $\Gamma$, let $\mathcal{A}(\Gamma)$ denote the smallest adequate set containing all formulas from $\Gamma$ and $h (\Gamma)$ denote the length of the longest branch in the corresponding tree of $\Gamma$. Put $\mathcal{B} (\Gamma, \mathcal{C}) := \{ \Box_i B \mid \Box_i B \in \mathcal{A}(\Delta) \cup \mathcal{C}, i \in \omega \}$. 

We define the strict partial order on the set of subsets of $\mathcal{B} (\Gamma, \mathcal{C})$.
Let $S_1 \prec S_2$ if there exists a natural number $j$ such that
\begin{itemize}
\item for any formula $A$ and any $i<j$ \[\Box_i A \in S_1 \Longleftrightarrow \Box_i A \in S_2;\]
\item for any formula $A$ \[\Box_j A \in S_1 \Longrightarrow \Box_j A \in S_2;\]
\item there exists a formula $\Box_j A$ such that \[\Box_j A \in S_2, \qquad \Box_j A \nin S_1.\]
\end{itemize}
Denote by $l(\Gamma, \mathcal{C})$ the size of the largest chain in the given partial order. It can be shown that $l(\Gamma, \mathcal{C}) = \prod^{m}_{i=0} (\lvert \mathcal{B}_i (\Gamma, \mathcal{C}) \rvert + 1)$, where $\mathcal{B}_i (\Gamma, \mathcal{C}) := \{ \Box_i B \mid \Box_i B \in \mathcal{A}(\Gamma) \cup \mathcal{C} \}$ and $m$ is the number of the maximal modality that occurs in $\mathcal{A}(\Gamma) \cup \mathcal{C}$.  

Now assume $\pi$ is an annotated proof of $\Gamma$ in $\mathsf{GLP_{NS}} + \mathsf{cut}(\mathcal{C}) + \boxplus\text{-}\mathsf{cut}_d(\mathcal{C})$, where $d=h(\Gamma) + l(\Gamma, \mathcal{C}) $. 
Consider any application of $\boxplus\text{-}\mathsf{cut}_d(\mathcal{C})$ in $\pi$
\begin{gather}
\label{App} 
\AXC{$\Delta \{\diamondtimes_i \overline{A} \} \{ \diamondtimes_i \overline{A}\}^n$}
\AXC{$\Delta \{\Box_i A \} \{ \emptyset\}^n$} 
\LeftLabel{$\boxplus\text{-}\mathsf{cut}_d(\mathcal{C})$}
\RightLabel{$( A \in \mathcal{C})$.}
\BIC{$\Delta \{\emptyset\} \{ \emptyset\}^n$}
\DisplayProof
\end{gather}
We will find a subproof $\pi_0$ of $\pi$ that contains this application of $\boxplus\text{-}\mathsf{cut}_d(\mathcal{C})$ and replace $\pi_0$ by a cut-free proof of the same sequent. Hence, by repeating this procedure for all other applications of $\boxplus\text{-}\mathsf{cut}_d(\mathcal{C})$, we will obtain an ordinary proof of $\Gamma$ in $\mathsf{GLP_{NS}} + \mathsf{cut}(\mathcal{C})$.   

Denote by $b$ the node of the first hole in the corresponding tree of $\Delta \{\emptyset\} \{ \emptyset\}^n$.
Let 
$$ r \rightarrow a_1 \rightarrow \dotso \rightarrow a_{h(\Gamma)} \rightarrow_{j_1} a_{h(\Gamma)+1} \rightarrow_{j_2}\dotso \to_{j_{l(\Gamma, \mathcal{C})}} a_{h(\Gamma)+ l(\Gamma, \mathcal{C})} \to \dotso \to b , $$
be the path in the corresponding tree of $\Delta \{\emptyset\} \{ \emptyset\}^n$ connecting the root $r$ and the node $b$. Notice that nodes $a_{h(\Gamma)+1}, a_{h(\Gamma)+2}, \dotso, a_{h(\Gamma)+ l(\Gamma, \mathcal{C})}$ occur deeper than all nodes in $\Gamma$. Thus, in the proof $\pi$ we can trace these nodes downwards via their descendants to the corresponding modal formulas introduced by the rule $\Box$. Let the node $a_{h(\Gamma)+k}$ is used to introduce the formula $\Box_{j_{k}} B_k$. Then the node $a_{h(\Gamma)+k}$ contains the formula $\Diamond_{j_{k}} \overline{B}_k $ from the definition of the rule $\Box$.

We claim that there exists a pair of different nodes $a_{h(\Gamma)+u}$ and $a_{h(\Gamma)+t}$ such that $0< u < t \leqslant l(\Gamma, \mathcal{C})$, $\Box_{j_{u}} B_u = \Box_{j_{t}} B_t$ and the path   
$$ a_{h(\Gamma)+u} \rightarrow_{j_{u+1}} \dotso \to_{j_{t-1}} a_{h(\Gamma)+t-1} $$
is a $j_u$-path, i.e. all arrows of these path are marked by natural numbers greater or equal than $j_u$.
For $i= 1,\dotsc, l(\Gamma, \mathcal{C})$, define $$S_{i} := \{ \Box_{j_{k}} B_k \mid k\leqslant i \text{ and the path from $a_{h(\Gamma)+k}$ to $a_{h(\Gamma)+i}$ is a $j_k$-path }\}.$$ 
Notice that $S_i$ are subsets of $\mathcal{B} (\Gamma, \mathcal{C})$ and if $\Box_{j_{k}} B_{k} \nin S_{k-1}$, then $S_{k-1} \prec S_{k}$.  
Suppose $\Box_{j_{k}} B_{k} \nin S_{k-1}$ for all $k=2, \dotsc, l(\Gamma, \mathcal{C})$. Then we see that $S_1 = \{ \Box_{j_1} B_1\}$ and there exists a chain 
\[\emptyset \prec S_1 \prec \dotso \prec S_{l(\Gamma, \mathcal{C})}\] 
of the size greater than $l(\Gamma, \mathcal{C})$, a contradiction with the definition of $l(\Gamma, \mathcal{C})$.
Hence, we have $\Box_{j_{t}} B_{t} \in S_{t-1}$ for some $t$, where $2 \leqslant t\leqslant l(\Gamma, \mathcal{C}) $.
By definition of $S_{t-1}$, there is a node $a_{h(\Gamma)+u}$ such that $0< u < t \leqslant l(\Gamma, \mathcal{C})$, $\Box_{j_{u}} B_u = \Box_{j_{t}} B_t$ and the path   
$$ a_{h(\Gamma)+u} \rightarrow_{j_{u+1}} \dotso \to_{j_{t-1}} a_{h(\Gamma)+t-1} $$
is a $j_u$-path

Now consider the application of the modal rule $\Box$ in $\pi$, where the descendant of $a_{h(\Gamma)+t}$ is used to introduce the formula  $\Box_{j_{t}} B_t$,
\[
\AXC{$\pi_0$}
\noLine
\UIC{\vdots}
\noLine
\UIC{$\Sigma \{[B_t, \Diamond_{j_{t}} \overline{B}_k ]_{j_{t}}  \}$} 
\LeftLabel{$\Box$}
\RightLabel{ .}
\UIC{$\Sigma \{ \Box_{j_{t}} B_t  \} $}
\DisplayProof 
 \]
Note that the application \eqref{App} of $\boxplus\text{-}\mathsf{cut}_d(\mathcal{C})$ occurs in $\pi_0$. Recall that the descendant of $a_{h(\Gamma)+u}$ in the corresponding tree of $\Sigma \{ \Box_{j_{t}} B_t  \} $ contains the formula $\Diamond_{j_{u}} \overline{B}_u$ and $\Diamond_{j_{u}} \overline{B}_u = \Diamond_{j_{t}} \overline{B}_t$. Consider the proof
\begin{gather}\label{cut-free proof}
\AXC{$\mu$}
\noLine
\UIC{\vdots}
\noLine
\UIC{$\Sigma^\prime \{[\Box_{j_{t}} B_t, \Diamond_{j_{t}} \overline{B}_t ]_{j_{t}}  \}$} 
\doubleLine
\LeftLabel{$\mathsf{tran}^\star$}
\RightLabel{ ,}
\UIC{$\Sigma \{ \Box_{j_{t}} B_t  \} $}
\DisplayProof 
\end{gather}
where $\mu$ is a cut-free proof of $\Sigma^\prime \{[\Box_{j_{t}} B_t, \Diamond_{j_{t}} \overline{B}_t ]_{j_{t}}  \}$ and $\Sigma \{ \Box_{j_{t}} B_t  \} $ is obtained from $\Sigma^\prime \{[\Box_{j_{t}} B_t, \Diamond_{j_{t}} \overline{B}_t ]_{j_{t}}  \}$ by applying the rule $\mathsf{tran}$ to the formula $\Diamond_{j_{t}} \overline{B}_t $ along the path of descendants of $a_{h(\Gamma)+u}, \dotso, a_{h(\Gamma)+t-1}$ in the corresponding tree of $\Sigma \{ \Box_{j_{t}} B_t  \} $. Now we can replace the subproof $\pi_0$ of $\pi$ by the cut-free proof \eqref{cut-free proof}. By repeating this procedure for all applications of $\boxplus\text{-}\mathsf{cut}_d(\mathcal{C})$ in $\pi$, we obtain an ordinary proof of $\Gamma$ in $\mathsf{GLP_{NS}} + \mathsf{cut}(\mathcal{C})$.

\end{proof}


\begin{cor} \label{Lob}
If there is an annotated proof of $\Gamma$ in $\mathsf{GLP_{NS}} + \mathsf{cut}(\mathcal{C}) + \boxplus\text{-}\mathsf{cut}(\mathcal{C})$, then there is an ordinary proof of $\Gamma$ in $\mathsf{GLP_{NS}} + \mathsf{cut}(\mathcal{C})$.
\end{cor}
\begin{proof}
Suppose $\mathsf{GLP_{NS}} + \mathsf{cut}(\mathcal{C}) + \boxplus\text{-}\mathsf{cut}(\mathcal{C}) \vdash \Gamma$. Then there exists a natural number $d_0$ such that $\mathsf{GLP_{NS}} + \mathsf{cut}(\mathcal{C}) + \boxplus\text{-}\mathsf{cut}_{d_0}(\mathcal{C}) \vdash \Gamma$. By Lemma \ref{sasaki lemma 1},  for all $d \geqslant d_0$ we have $\mathsf{GLP_{NS}} + \mathsf{cut}(\mathcal{C}) + \boxplus\text{-}\mathsf{cut}_{d}(\mathcal{C}) \vdash \Gamma$. Applying Lemma \ref{sasaki lemma 2}, we obtain $\mathsf{GLP_{NS}} + \mathsf{cut}(\mathcal{C}) \vdash \Gamma$.  
\end{proof}

From the previous corollary and Lemma \ref{reduction lemma}, we have:

\begin{cor} \label{reduction lemma 2} 
The rule $\mathsf{cut}(A)$ is admissible for $\mathsf{GLP_{NS}} + \mathsf{cut}(\mathcal{C}_A) $.
\end{cor}

We now in a position to prove the cut elimination theorem.   
\begin{thm}[Cut Elimination] \label{cut elim}
If $\mathsf{GLP_{NS} + cut}\vdash \Gamma$, then $\mathsf{GLP_{NS}}\vdash \Gamma$.
\end{thm}
\begin{proof}
Assume we have
\[  
\AXC{$\pi_1$}
\noLine
\UIC{\vdots}
\noLine
\UIC{$\Sigma \{ A\}$}
\AXC{$\pi_2$}
\noLine
\UIC{\vdots}
\noLine
\UIC{$\Sigma \{ \overline{A}\}$} 
\LeftLabel{$\mathsf{cut}$}
\RightLabel{ ,}
\BIC{$\Sigma \{\emptyset\} $}
\DisplayProof
 \]
where $\pi_1$ and $\pi_2$ are proofs in $\mathsf{GLP_{NS}}$. 
By induction on $\rvert A \lvert $, we prove $\mathsf{GLP_{NS}}  \vdash \Sigma \{\emptyset\} $. From Corollary \ref{reduction lemma}, we have $\mathsf{GLP_{NS}} + \mathsf{cut}(\mathcal{C}_A)   \vdash \Sigma \{\emptyset\}$. For any formula $B$ from $\mathcal{C}_A$, we have $\rvert B \lvert < \rvert A \lvert$. Thus, the induction hypothesis implies $\mathsf{GLP_{NS}}    \vdash \Sigma \{\emptyset\}$.

\end{proof}
\section{An application}
In the present section we establish the reduction of $\mathsf{GLP}$ to its fragment called $\mathsf{J}$ via the cut elimination theorem.
  
Recall that the logic $\mathsf{J}$ is a fragment of $\mathsf{GLP}$ obtained by replacing axiom (v) by the following two axioms derivable in $\mathsf{GLP}$:
\begin{itemize}
\item[(vi)] $\Box_i A \rightarrow \Box_j \Box_i A$ for $i \leqslant j$;
\item[(vii)] $\Box_i A \rightarrow \Box_i \Box_j A$ for $i \leqslant j$.
\end{itemize}

By $m(A)$, denote the number of the maximal modality that occurs in $A$. If $A$ does not contain any modality, then we put $m(A)= -1$.
For a given formula $A$ let 
\[ M(A):= \bigwedge_{k<s} \bigwedge_{i_k < j \leqslant m(A)} (\Box_{i_k} A_k \rightarrow \Box_{j} A_k  ), \]
where $\Box_{i_k} A_k$ for $k<s$ are all subformulas of $A$ (and of $\overline{A}$) of the form $\Box_i B$. Set
\[
M^+(A):= M(A) \wedge \bigwedge_{i \leqslant m(A)} \Box_i M(A).\]
\begin{thm} \label{J-GLP}
$\mathsf{J} \vdash M^{+} (A) \rightarrow  A \Longleftrightarrow \mathsf{GLP} \vdash A$.
\end{thm}
The first proof of the theorem was given in \cite{Bek10a} by providing a complete Kripke semantics for $\mathsf{GLP}$.\footnote{In \cite{Bek10a}, the formula $M (A)$ was misstated. The correct version was given in \cite{Bek11, BG13}.}
Other proofs, on the basis of arithmetical semantics and topological semantics, were given in \cite{Bek11, BG13}.
Now we present a new proof of this fact on the basis of the cut elimination theorem.

Let us define the sequent system $\mathsf{J_{NS}}$ for the logic $\mathsf{J}$. Initial sequents and logical rules of $\mathsf{J_{NS}}$ have the same form as of $\mathsf{GLP_{NS}}$, and the modal rules are the following:
\begin{gather*}
\AXC{$\Gamma \{ [A, \Diamond_i \overline{A}]_i \}$}
\LeftLabel{$\mathsf{\Box}$}
\UIC{$\Gamma \{ \Box_i A\}$}
\DisplayProof \quad 
\AXC{$\Gamma \{\Diamond_i A, [A, \Delta ]_i\}$}
\LeftLabel{$\mathsf{\Diamond^\prime}$}
\UIC{$\Gamma \{\Diamond_i A, [\Delta ]_i\}$}
\DisplayProof\\\\
\AXC{$\Gamma \{\Diamond_i A, [\Diamond_i A, \Delta ]_j\}$}
\LeftLabel{$\mathsf{tran}$}
\RightLabel{$(i\leqslant j)$}
\UIC{$\Gamma \{\Diamond_i A, [\Delta ]_j\}$}
\DisplayProof \quad
\AXC{$\Gamma \{\Diamond_i A, [\Diamond_j A, \Delta ]_i\}$}
\LeftLabel{$\mathsf{tran^\prime}$}
\RightLabel{$(i\leqslant j)$}
\UIC{$\Gamma \{\Diamond_i A, [\Delta ]_i\}$}
\DisplayProof\\\\
\AXC{$\Gamma \{\Diamond_i A, [\Diamond_i A, \Delta ]_j\}$}
\LeftLabel{$\mathsf{eucl}$}
\RightLabel{$(i < j)$}
\UIC{$\Gamma \{ [\Diamond_i A, \Delta ]_j\}$}
\DisplayProof \:.
\end{gather*}
\begin{lem}
$\mathsf{J_{NS}} + \mathsf{weak} + \mathsf{cont} \vdash \Gamma \Rightarrow \mathsf{J} \vdash  \Gamma^\sharp$.
\end{lem}
For a sequent $\Gamma$ let $m(\Gamma) = m(\Gamma^\sharp$).
Define 
\[ W(\Gamma):= \bigcup_{k<s} \bigcup_{i_k < j \leqslant m(\Gamma)} \{\Box_{i_k} A_k \wedge \overline{\Box_{j} A_k}  \}, \]
where $\Box_{i_k} A_k$ for $k<s$ are all subformulas of $\Gamma^\sharp$ (and of $\overline{\Gamma^\sharp}$) of the form $\Box_i B$. Set
\begin{gather*}
\Diamond_i W(\Gamma) := \{\Diamond_i B \colon B \in W(\Gamma) \},\quad
W^+(\Gamma):= W(\Gamma) \cup \bigcup_{i \leqslant m(\Gamma)} \Diamond_i W(\Gamma) .
\end{gather*}
Given a sequent $\Gamma$, let us consider the corresponding tree of $\Gamma$ and extend all multisets of formulas in the nodes of \emph{tree}($\Gamma$) by the multiset $W^+ (\Gamma)$. The result of the procedure is denoted by $\Gamma^\ast$. Note that $\{ A \}^{\ast \sharp}$ is equivalent with $M^{+} (A) \rightarrow  A$ in $\mathsf{J}$.

\begin{lem}\label{GLP-J-Seq}
$\mathsf{GLP_{NS}} \vdash \Gamma \Rightarrow \mathsf{J_{NS}}  + \mathsf{weak} + \mathsf{cont} \vdash  \Gamma^\star$.
\end{lem}
\begin{proof}
Assume $\pi$ is a proof of $\Gamma$ in $\mathsf{GLP_{NS}} $. We prove $\mathsf{J_{NS}}  + \mathsf{weak} + \mathsf{cont} \vdash  \Gamma^\star$ by induction on $\lvert \pi \rvert$. If $\Gamma$ is an initial sequent, than $\Gamma^\star$ is also an initial sequent and $\mathsf{J_{NS}}  + \mathsf{weak} + \mathsf{cont} \vdash  \Gamma^\star$. Otherwise, consider the lowermost application of an inference rule in $\pi$.


Case 1. The lowermost inference has the form:
\[\AXC{$\Sigma \{ A \}$}
\AXC{$\Sigma \{ B \}$}
\LeftLabel{$\mathsf{\wedge}$}
\RightLabel{ .}
\BIC{$\Sigma \{ A \wedge B \}$}
\DisplayProof
\]
We see $W^+ (\Gamma) = W^+ (\Sigma \{ A \wedge B \})$, $W^+ (\Gamma) \supset W^+ (\Sigma \{ A \})$ and $W^+ (\Gamma) \supset W^+ (\Sigma \{ B \})$. We extend all multisets in the nodes of \emph{tree}($\Sigma \{ \:\}$) by the multiset $W^+ (\Gamma)$ and denote the result by $\Delta \{\:\}$. We have $(\Sigma \{  A \wedge B  \})^\ast = \Delta \{  A \wedge B  \}$. By the induction hypotheses, sequents $(\Sigma \{ A B \})^\ast$ and $(\Sigma \{ B \})^\ast$ are provable in $\mathsf{J_{Seq}}  + \mathsf{weak} + \mathsf{cont}$. We get 
\[\AXC{$(\Sigma \{ A \})^\ast$}
\LeftLabel{$\mathsf{weak}$}
\UIC{$\Delta \{A\}$}
\AXC{$(\Sigma \{ B \})^\ast$}
\LeftLabel{$\mathsf{weak}$}
\UIC{$\Delta \{B\}$}
\LeftLabel{$\mathsf{\wedge}$}
\RightLabel{ .}
\BIC{$\Delta \{ A \wedge B \}$}
\DisplayProof
\]
Hence, the sequent $(\Sigma \{ A \wedge B \})^\ast = \Delta \{A \wedge B \}$ is provable in $\mathsf{J_{NS}}  + \mathsf{weak} + \mathsf{cont}$.

Case 2. The lowermost application of an inference rule in $\pi$ has the form:
\[
\AXC{$\Sigma \{ A,B \}$}
\LeftLabel{$\mathsf{\vee}$}
\RightLabel{ .}
\UIC{$\Sigma \{ A \vee B \}$}
\DisplayProof
\]
We see $W^+ (\Gamma) = W^+ (\Sigma \{ A \vee B \}) = W^+ (\Sigma \{ A, B \})$. We extend all multisets in the nodes of \emph{tree}($\Sigma \{ \:\}$) by the multiset $W^+ (\Gamma)$ and denote the result by $\Delta \{\:\}$. We have $(\Sigma \{  A \vee B  \})^\ast = \Delta \{  A \vee B  \}$ and $(\Sigma \{  A, B  \})^\ast = \Delta \{  A, B  \}$. By the induction hypotheses, the sequent $(\Sigma \{ A, B \})^\ast$, which is equal to $\Delta \{  A, B  \}$, is provable in $\mathsf{J_{Seq}}  + \mathsf{weak} + \mathsf{cont}$. We get 
\[
\AXC{$\Delta \{ A,B \}$}
\LeftLabel{$\mathsf{\vee}$}
\RightLabel{ .}
\UIC{$\Delta \{ A \vee B \}$}
\DisplayProof
\]
Hence, the sequent $(\Sigma \{ A \vee B \})^\ast = \Delta \{A \vee B \}$ is provable in $\mathsf{J_{NS}}  + \mathsf{weak} + \mathsf{cont}$.

Case 3. The lowermost application of an inference rule in $\pi$ has the form: 
\[\AXC{$\Sigma \{ [A, \Diamond_i \overline{A}]_i \}$}
\LeftLabel{$\mathsf{\Box}$}
\UIC{$\Sigma \{ \Box_i A\}$}
\DisplayProof \;.
\] 
We see $W^+ (\Gamma) = W^+ (\Sigma \{ [A, \Diamond_i \overline{A}]_i \}) = W^+ (\Sigma \{ \Box_i A\})$. We extend all multisets in the nodes of \emph{tree}($\Sigma \{ \:\}$) by the multiset $W^+ (\Gamma)$ and denote the result by $\Delta \{\:\}$. We have $(\Sigma \{ [A, \Diamond_i \overline{A}]_i \})^\ast = \Delta \{ [ W^+ (\Gamma), A, \Diamond_i \overline{A}]_i \}$ and $(\Sigma \{ \Box_i A\})^\ast = \Delta \{ \Box_i A\}$. 
By the induction hypotheses, the sequent $(\Sigma \{ [ A, \Diamond_i \overline{A}]_i \} )^\ast$, which is equal to $ \Delta \{ [ W^+ (\Gamma), A, \Diamond_i \overline{A}]_i\}$, is provable in $\mathsf{J_{Seq}}  + \mathsf{weak} + \mathsf{cont}$. We have 
\[\AXC{$\Delta \{ [ W^+ (\Gamma), A, \Diamond_i \overline{A}]_i \}$}
\doubleLine
\LeftLabel{$ \mathsf{{\Diamond^\prime}^\ast},\mathsf{tran^\ast},\mathsf{{tran^\prime}^\ast} $}
\RightLabel{ .}
\UIC{$\Delta \{ [ A, \Diamond_i \overline{A}]_i \}$}
\LeftLabel{$\mathsf{\Box}$}
\UIC{$\Delta \{ \Box_i A\}$}
\DisplayProof \] 
Hence, the sequent $(\Sigma \{ \Box_i A\})^\ast = \Delta \{ \Box_i A\}$ is provable in $\mathsf{J_{NS}}  + \mathsf{weak} + \mathsf{cont}$.

Case 4. The last application of an inference rule in $\pi$ has the form:
\[\AXC{$\Sigma \{\Diamond_i A, [A, \Upsilon ]_j\}$}
\LeftLabel{$\mathsf{\Diamond}$}
\RightLabel{$(i \leqslant j)$.}
\UIC{$\Sigma \{\Diamond_i A, [\Upsilon ]_j\}$}
\DisplayProof\]
We see $W^+ (\Gamma) = W^+ (\Sigma \{\Diamond_i A, [\Upsilon ]_j\}) = W^+ (\Sigma \{\Diamond_i A, [A,\Upsilon ]_j\})$. We extend all multisets in the nodes of \emph{tree}($\Sigma \{ \:\}$) by the multiset $W^+ (\Gamma)$ and denote the result by $\Delta \{\:\}$. We have $(\Sigma \{\Diamond_i A, [A, \Upsilon ]_j\} \})^\ast = \Delta \{ \Diamond_i A, [  A, \Upsilon^\ast]_j \}$ and $(\Sigma \{\Diamond_i A, [ \Upsilon ]_j\})^\ast = \Delta \{ \Diamond_i A, [ \Upsilon^\ast]_j\}$. 
By the induction hypotheses, the sequent $(\Sigma \{\Diamond_i A, [A, \Upsilon ]_j\} \})^\ast$, which is equal to $\Delta \{ \Diamond_i A, [  A, \Upsilon^\ast]_j$, is provable in $\mathsf{J_{Seq}}  + \mathsf{weak} + \mathsf{cont}$. If $i=j$, then we get
\[
\AXC{$\Delta \{ \Diamond_i A, [  A, \Upsilon^\ast]_i \}$}
\LeftLabel{$\mathsf{\Diamond^\prime}$}
\RightLabel{ .}
\UIC{$\Delta \{ \Diamond_i A, [ \Upsilon^\ast]_i\}$}
\DisplayProof\]
Otherwise, we have $i<j$ and
\[\AXC{$\mathsf{Ax}$}
\noLine
\UIC{$\Delta \{\Diamond_i A, \Box_i \overline{A} ,  [\Upsilon^\ast ]_j\}$}
\AXC{$\Delta \{ \Diamond_i A, [  A, \Upsilon^\ast]_j$}
\LeftLabel{$\mathsf{weak}$}
\UIC{$\Delta \{ \Diamond_i A, \Diamond_j A,[  A, \Upsilon^\ast]_j$}
\LeftLabel{$\mathsf{\Diamond^\prime}$}
\UIC{$\Delta \{ \Diamond_i A, \Diamond_j A,[   \Upsilon^\ast]_j$}
\LeftLabel{$\mathsf{\wedge}$}
\BIC{$\Delta \{ \Diamond_i A, \Box_i \overline{A} \wedge \Diamond_j A, [\Upsilon^\ast ]_j\}$}
\LeftLabel{$\mathsf{cont}$}
\RightLabel{$(\Box_i \overline{A} \wedge \Diamond_j A \in W(\Gamma))$.}
\UIC{$\Delta \{ \Diamond_i A, [\Upsilon^\ast ]_j\}$}
\DisplayProof  \] 
Hence, the sequent $(\Sigma \{\Diamond_i A, [ \Delta ]_j\})^\ast = \Delta \{ \Diamond_i A, [ \Upsilon^\ast]_j\}$ is provable in $\mathsf{J_{NS}}  + \mathsf{weak} + \mathsf{cont}$.

The remaining cases of the rules $\mathsf{tran}$ and $\mathsf{eucl}$ are completely analogous to the case 2, so we omit them. 
\end{proof}

\begin{proof}[Proof of Theorem \ref{J-GLP}]
For any formula $A$, $\mathsf{GLP} \vdash M^{+} (A)$. Thus the left-to-right part is obvious. 
Prove the converse. 

If $\mathsf{GLP} \vdash A$, then $\mathsf{GLP_{NS}} + \mathsf{cut} \vdash A$. By Theorem \ref{cut elim}, this yields $\mathsf{GLP_{NS}} \vdash A$. Applying Lemma \ref{GLP-J-Seq}, we have $\mathsf{J_{NS}}  + \mathsf{weak} + \mathsf{cont} \vdash \{A\}^\ast$. Hence, $\mathsf{J} \vdash M^{+} (A) \rightarrow  A$.
\end{proof}
\paragraph*{Acknowledgements.} I would like to thank Kai Br\"{u}nnler for introducing me to the calculus of nested sequents. Additionally, I am thankful to my wife Maria Shamkanova for her personal support during the work on the paper.

\printbibliography

\end{document}